\documentclass[12pt]{article}  

\usepackage{fullpage}   
\usepackage{amsmath}
\usepackage{amsfonts}
\usepackage{amssymb}
\usepackage{amsthm}
\usepackage{graphicx}
\usepackage{color}
\usepackage{pinlabel}
\usepackage{hyperref}
\usepackage[textwidth=6in]{geometry}

\newtheorem{theorem}{Theorem}[section]
\newtheorem{proposition}[theorem]{Proposition}

\newtheorem{corollary}[theorem]{Corollary}
\theoremstyle{definition}
\newtheorem{definition}[theorem]{Definition}

\theoremstyle{remark}
\newtheorem{remark}[theorem]{Remark}
\newtheorem{example}[theorem]{Example}

\newcommand{\ba}{\setminus}

\newcommand{\la}{\lambda}
\newcommand{\tcon}{\rightthreetimes}
\newcommand{\ident}{\odot}

  \DeclareMathOperator{\cross}{cr}
\DeclareMathOperator{\white}{wh}
\DeclareMathOperator{\black}{bl}
\DeclareMathOperator{\total}{tot}

\begin{document}

\title{Edge colourings and topological graph polynomials}

\author{Joanna A. Ellis-Monaghan  \\  \textit{Department of Mathematics} \\ \textit{Saint Michael's College}\\  \textit{USA} \\ \texttt{jellis-monaghan@smcvt.edu} \\
\and
Louis H. Kauffman \\ \textit{Department of Mathematics, Statistics and Computer Science} \\  \textit{University of Illinois at Chicago}\\ 
\textit{USA}\\
\textit{and}\\
\textit{Department of Mechanics and Mathematics}\\
\textit{Novosibirsk State University}\\
\textit{Russia}
 \\ \texttt{kauffman@uic.edu} \\
\and
Iain Moffatt \\ \textit{Department of Mathematics} \\
\textit{Royal Holloway,
University of London}
\\ 
\textit{United Kingdom} \\
\texttt{iain.moffatt@rhul.ac.uk}
}




\date{}


\maketitle


\begin{abstract}
A $k$-valuation is a special type of edge $k$-colouring of a medial graph. Various graph polynomials, such as the Tutte, Penrose, Bollob\'as--Riordan, and transition polynomials, admit combinatorial interpretations and evaluations as  weighted counts of $k$-valuations.  In this paper, we consider a multivariate generating function of $k$-valuations. We show that this is a polynomial in $k$ and hence defines a graph polynomial. We then show that the resulting polynomial has several desirable properties,  including a recursive deletion-contraction-type definition, and specialises to the graph polynomials mentioned above. It also offers an alternative extension of the Penrose polynomial from plane graphs to graphs in other surfaces. 
\end{abstract}

\section{Introduction}\label{dyh}
Graph polynomials, including the many recent topological graph polynomials, encode combinatorial information.  The challenge lies in extracting this information. 
We present a new topological graph polynomial, and show how it counts edge colourings. We carry over this combinatorial interpretation to some other well-known topological graph polynomials. This new graph polynomial was motivated by work of Penrose, as follows.

Penrose's highly influential paper \cite{MR0281657} gives a graphical calculus for computing the number of proper edge 3-colourings of plane graphs. 
The work \cite{MR3665180}  built upon Penrose's by extending this graphical calculus to give a method for counting edge 3-colourings of any (not necessarily planar) cubic graph. This calculus is applied to an immersion of the graph in the plane, but it does not depend upon the particular immersion chosen. 

On the other hand, Penrose's graphical calculus led to the introduction of the Penrose polynomial, $P(G;\lambda)$, of plane graphs \cite{MR1428870}. This was extended to graphs in any surface in \cite{MR2994409}. 
For cubic \emph{plane} graphs, the Penrose polynomial evaluated at $\lambda=3$ gives the number of proper edge 3-colourings. This result, however, does not hold for graphs embedded in surfaces of higher genus. 

Thus Penrose's calculus for counting edge 3-colourings of cubic planar graphs has two different topological extensions: one polynomial invariant of embedded graphs that counts edge 3-colourings only in the plane case, 
the other a generalised Penrose graphical calculus for immersed graphs that counts edge 3-colourings for any cubic plane graph via an immersion of it in the plane~\cite{MR3665180}. 
Together, these extensions hint at the existence of  a polynomial invariant of graphs in surfaces that counts edge 3-colourings of all cubic  graphs, not just plane ones. 

In this paper we find such a polynomial. We approach this via a
 generating function of $k$-valuations, which are  special kinds of edge colourings.  We show that this generating function is a polynomial in $k$.  
 We relate the resulting topological graph  polynomial $\Omega$  to the Penrose polynomial, as well as to other topological graph polynomials, and use it to find   combinatorial information in topological graph polynomials.

The $k$-valuations central to this investigation are colourings of medial graphs. 
The construction of the medial graph $G_m$ of an embedded graph $G$ is well-known in graph theory. However, there are different ways in the literature for constructing $G_m$ when $G$ has isolated vertices. Here we allow  medial graphs to have `free-loops', which arise from isolated vertices. (See Remark~\ref{remfl} for a discussion of the case where free-loops are not considered, where it is seen that the results presented here are easily adapted to this situation.)  A \emph{free-loop} is a circular edge in a graph that has no incident vertex. The definition of a medial graph we use here is as follows. 
  Let $G$ be a graph embedded in a surface $\Sigma$. The \emph{medial graph}, $G_m$, of $G$ is the 4-regular graph (with free-loops) embedded in $\Sigma$ obtained by placing a vertex on each edge of $G$ then adding edges as curves on the surface that follow the face boundaries between vertices. If $G$ has isolated vertices, then, for each isolated vertex, add a free-loop as a curve that follows its boundary (or technically, the boundary of a regular neighbourhood of the vertex).
The graph $G$ does not form part of the medial graph. 
Although medial graphs are most commonly studied in the setting of plane graphs, they are defined for graphs embedded in any surface, including those with boundary, or with non-cellular embeddings.

Each vertex of  $G \subset \Sigma$  corresponds to a unique face of $G_m \subset \Sigma$. (Here by a  \emph{face} we mean a component of $\Sigma \ba G_m$.) This correspondence gives rise to a  natural face 2-colouring of $G_m$ which is obtained by colouring the  faces of $G_m$ that correspond to vertices of $G$  black, and the remaining faces  white. This is called the \emph{canonical checkerboard colouring} of $G_m$. 

Let $k$ be a natural number.  A  \emph{$k$-valuation} of $G_m$  is an edge $k$-colouring of $G_m$ such that each vertex is incident to an even number (possibly zero) of edges of each colour.  Here we  denote $k$-valuations by $\phi$ and consider them as mappings from  to $[k]=\{1,\ldots, k\}$. 

Since several important  graph polynomials have been shown to count $k$-valuations, they play an important role in the theory of topological graph polynomials. To describe these interpretations we need a little more terminology.

A $k$-valuation of a canonically checkerboard coloured medial graph $G_m$ yields four possible \emph{configurations} of colours at each  vertex, which we term \emph{white}, \emph{black}, \emph{crossing}, or \emph{total},  as in Figure~\ref{f.kval}. 

\begin{figure}[ht]
\centering
\begin{tabular}{ccccccc}
\labellist \small\hair 2pt
\pinlabel {$i$}  at 23 64
\pinlabel {$j$}  at 23 9
\pinlabel {$j$}  at 48 9
\pinlabel {$i$}  at 48 64
\endlabellist
 \includegraphics[width=1.2cm]{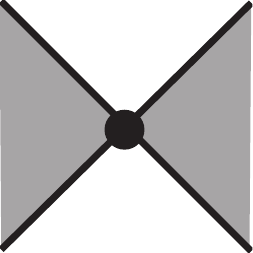}
  &  \quad\quad\quad &
  \labellist \small\hair 2pt
\pinlabel {$i$}  at 23 64
\pinlabel {$i$}  at 23 9
\pinlabel {$j$}  at 48 9
\pinlabel {$j$}  at 48 64
\endlabellist
 \includegraphics[width=1.2cm]{m1}
  & \quad\quad\quad&
  \labellist \small\hair 2pt
\pinlabel {$i$}  at 23 64
\pinlabel {$j$}  at 23 9
\pinlabel {$i$}  at 48 9
\pinlabel {$j$}  at 48 64
\endlabellist
 \includegraphics[width=1.2cm]{m1}
  & \quad\quad\quad&
  \labellist \small\hair 2pt
\pinlabel {$i$}  at 23 64
\pinlabel {$i$}  at 23 9
\pinlabel {$i$}  at 48 9
\pinlabel {$i$}  at 48 64
\endlabellist
 \includegraphics[width=1.2cm]{m1},
  \\
white && black &&crossing &&  total
\end{tabular}
\caption{Classifying $k$-valuations at a vertex, where $i\neq j$}
\label{f.kval}
\end{figure}

We let $\white(\phi)$, $\black(\phi)$,  $\cross(\phi)$ and $\total(\phi)$ denote the numbers of white vertices, black vertices,  crossing vertices and total vertices, respectively, in a $k$-valuation $\phi$.

As noted above, several graph polynomials have combinatorial interpretations as counts of $k$-valuations. The following theorem summarises what is known. In the theorem, $T(G;x,y)$ denotes the Tutte polynomial \cite{MR0061366}, $R(G;x,y,z)$ the Bollob\'as--Riordan polynomial \cite{MR1851080,MR1906909}, $P(G;\la)$  the (topological) Penrose polynomial \cite{MR1428870,MR2994409,MR0281657}, and $  Q(G; ( \alpha, \beta , \gamma ), t) $  the topological transition polynomial \cite{MR2869185,MR1096990}. We do not recall the definitions of these polynomials here since we do not  need their details.

\begin{theorem}[\cite{MR1428870,MR3326433,MR2994409,MR1096990,KP}] 
\label{t1}
The following identities hold.
\begin{align}
\label{e.qk4} T(G; k/b+1 ,kb+1) &=  \sum   (b+1)^{\total(\phi)}     b^{\white(\phi)},
\\
\label{e.intro1}
 k^{\kappa(G)}  R(G;k+1,k,1/k,1) &= \sum 2^{\total(\phi)},
\\
\label{e.qk1}     k^{\kappa(G)}  b^{r(G)}   R\Big(G;  (k+b)/b, bk, 1/k\Big)    &= \sum   (b+1)^{\total(\phi)}     b^{\white(\phi)},
\\
\label{e.intro5}
P(G;k) &=  \# \text{Admissible $k$-valuations of } G_m,
\\
\label{e.intro2}
P(G;k) &= \sum  (-1)^{\cross(\phi)},
\\
\label{e.intro3}
P(G;-k) &= (-1)^{f(G)}  \sum  2^{\total(\phi)},
\\
\label{e.intro4}
  Q(G; ( \alpha, \beta , \gamma ), k) &= \sum   (\alpha+\beta+\gamma)^{\total(\phi)}  \alpha^{\white(\phi)}  \beta^{\black(\phi)}  \gamma^{\cross(\phi)}.
\end{align}
Here the sums in \eqref{e.qk4}--\eqref{e.qk1} are over all $k$-valuations $\phi$ of $G_m$ that contain no crossing configurations. An \emph{admissible} $k$-valuation is one that contains no black configurations, and the sums in \eqref{e.intro2} and \eqref{e.intro3} are over all admissible $k$-valuations of $G_m$. The sum in \eqref{e.intro4} is over all $k$-valuations $\phi$ of $G_m$. Furthermore,  in \eqref{e.qk4} and \eqref{e.intro5} $G$ must be plane, and in \eqref{e.intro3} the Petrie dual (defined in Section~\ref{sec2}) of the geometric dual of $G$ must be orientable for the identity to hold.
\end{theorem}

Equation~\eqref{e.intro1} is due to Korn and Pak \cite{KP}.
Equation~\eqref{e.intro4} Is due to   Jaeger \cite{MR1096990} and Aigner \cite{MR1428870}. 
The remaining interpretations are due to Ellis-Monaghan and Moffatt. Equation   \eqref{e.intro2} and a special case of \eqref{e.intro3} is  from   \cite{MR2994409}. The remaining identities are from \cite{MR3326433}. Equation~\eqref{e.qk1} generalises \eqref{e.qk4} and \eqref{e.intro1}, and \eqref{e.intro2} generalises \eqref{e.intro5}. Equations~\eqref{e.qk4}--\eqref{e.intro3} can all be obtained from \eqref{e.intro4}, as described in \cite{MR3326433}.

Equations~\eqref{e.qk4}--\eqref{e.intro4} address the problem of finding combinatorial interpretations of topological graph polynomials. 
In this paper we  invert the problem and take $k$-valuations as our starting point. 
\begin{definition}\label{d.2}
For an embedded graph $G$ and natural number $k$, let 
\[\Omega_k (G; w, x,y,z) :=\sum_{\phi  \text{ a $k$-valuation of }G_m}  w^{\total(\phi)}  x^{\white(\phi)}  y^{\black(\phi)}  z^{\cross(\phi)}   \]
where $\white(\phi)$, $\black(\phi)$,  $\cross(\phi)$ and $\total(\phi)$ denote the numbers of white vertices, black vertices,  crossing vertices and total vertices, respectively, in a $k$-valuation $\phi$.
 \end{definition}

Our approach to understanding  $\Omega_k $ is skein theoretic. We provide a recursive definition for  $\Omega_k $ (similar to the skein relations defining the transition polynomial, and akin to the deletion-contraction relations for the Tutte polynomial).  Doing this, however, requires us to consider $k$-valuations of a generalisation of embedded graphs that we name `edge-point ribbon graphs'. These are essentially ribbon graphs whose vertices can intersect in points. In the setting of edge-point ribbon graphs, we define, recursively, a polynomial invariant $\Omega(G) \in \mathbb{Z}[w,x,y,z,t]$, give a state-sum formulation of it, and prove that $\Omega_k (G)$ gives a combinatorial interpretation of it.  It follows from this that  $\Omega_k $  is a polynomial in $k$, and can be considered as a graph polynomial.
 We show $\Omega(G)$ specialises to the Tutte polynomial, the Bollob\'as--Riordan polynomial,  the (topological) Penrose polynomial, and the (topological) transition polynomial. Finally, we give some combinatorial evaluations of $\Omega$ and $\Omega_k$ in terms of graph colourings.

\begin{remark}\label{remfl}
Here  we have allowed medial graphs to have free-loops, which arise from isolated vertices. Medial graphs can also be constructed without reference to free-loops: just construct $G_m$ by placing a vertex on each edge of $G$ then adding edges as curves on the surface that follow the face boundaries between vertices. This construction ignores any isolated vertices, so, for example, the medial graph of any edgeless graph would be the empty graph. It is straightforward to adapt the results in this paper for this construction of medial graphs.
If we use $\widetilde{G}_m$ to denote the medial graph of $G$ constructed in this way without free-loops, and  $\widetilde{\Omega}_k(G)$ to denote the generating function of Definition~\ref{d.2}, but summing over  $k$-valuations of $\widetilde{G}_m$. Then 
$  k^{\ell} \, \widetilde{\Omega}_k(G) =   \Omega_k(G)$ when $G$ has $\ell$ isolated vertices.
\end{remark}

\section{Ribbon graphs}\label{sec2}
As is often the case when working with topological graph polynomials, it is convenient to describe embedded graphs as ribbon graphs.  
We recall some basic definitions about ribbon graphs here, and refer the reader to, for example, \cite{EMMbook} for additional background on them.

A {\em ribbon graph} $G=\left(V(G),E(G)\right)$ is a surface with boundary, represented as the union of two sets of discs --- a set $V(G)$ of {\em vertices} and a set $E(G)$ of {\em edges} --- such that: (1) the vertices and edges intersect in disjoint line segments; (2) each such line segment lies on the boundary of precisely one vertex and precisely one edge; and (3) every edge contains exactly two such line segments. 

Two ribbon graphs $G$ and $G'$ are \emph{equivalent} is there is a homeomorphism from $G$ to $G'$ (orientation preserving when $G$ is orientable) mapping $V(G)$ to $V(G')$, $E(G)$ to $E(G')$. In particular the   cyclic order of half-edges at each vertex is preserved.

Let $G$ be a ribbon graph and $e\in E(G)$. Then $G\ba e$ denotes the ribbon graph obtained from $G$ by \emph{deleting} the edge $e$.
If $u$ and $v$ are the (not necessarily distinct) vertices incident with $e$, then $G/e$ denotes the ribbon graph obtained as follows: consider the boundary component(s) of $e\cup\{u,v\}$ as curves on $G$. For each resulting curve, attach a disc (which will form a vertex of $G/e$) by identifying its boundary component with the curve. Delete $e$, $u$ and $v$ from the resulting complex, to get the ribbon graph $G/e$. We say $G/e$ is obtained from $G$ by {\em contracting} $e$. See Figure~\ref{tablecontractrg} for the local effect of contracting an edge of a ribbon graph.

\begin{table}[t]
\centering
\begin{tabular}{|c||c|c|c|}\hline
& non-loop & non-orientable loop & orientable loop \\ \hline
\raisebox{6mm}{$G$} 
&\includegraphics[scale=.25]{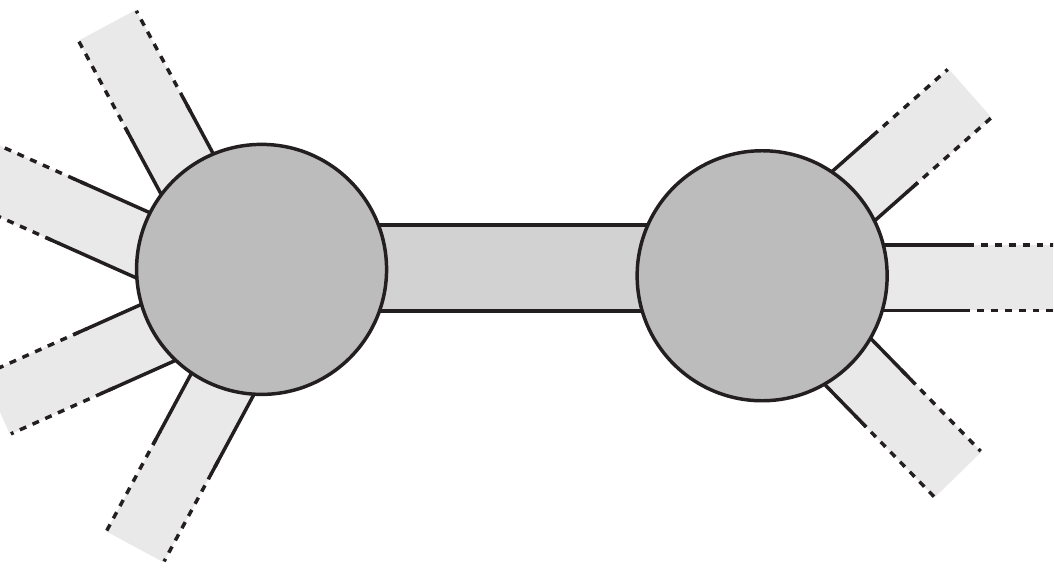} &\includegraphics[scale=.25]{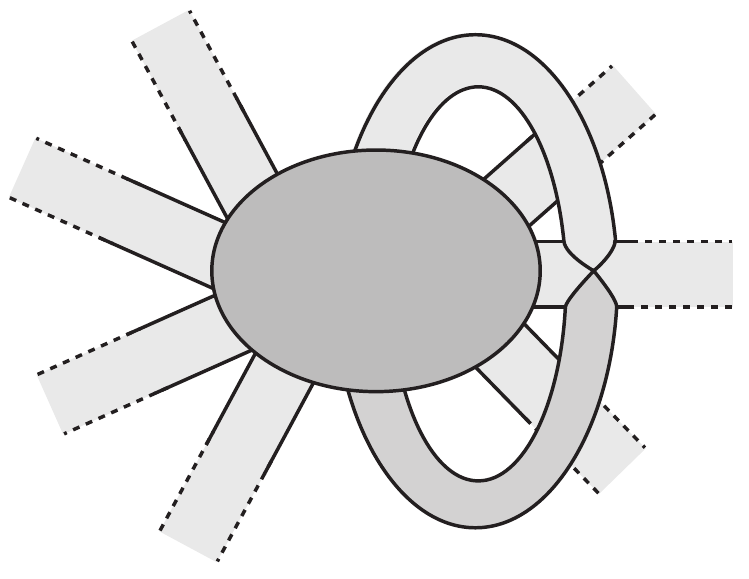} &\includegraphics[scale=.25]{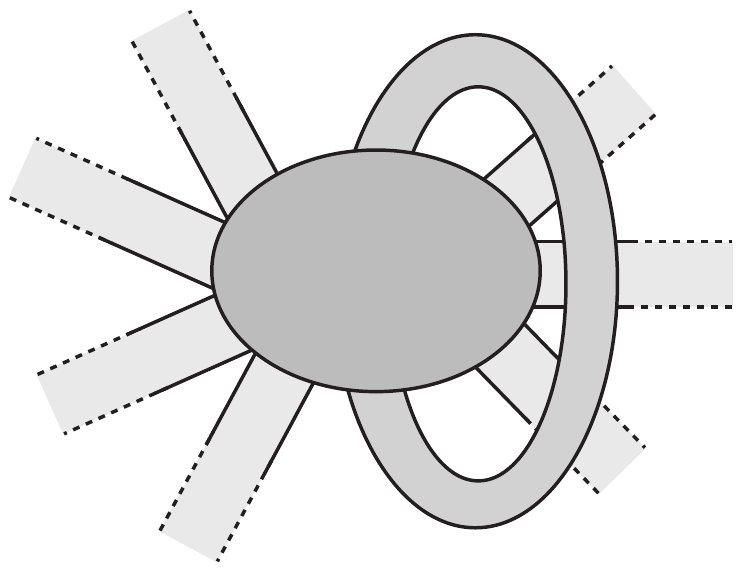}
\\ \hline
\raisebox{6mm}{$G/e$} 
&\includegraphics[scale=.25]{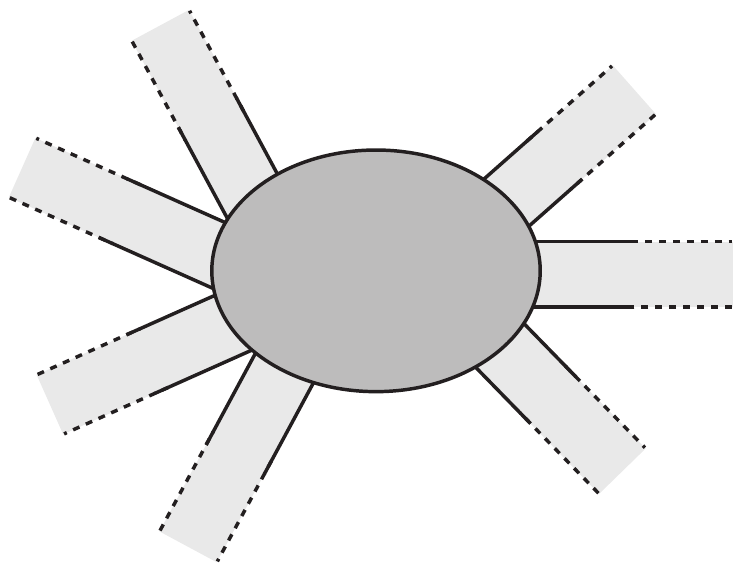} &\includegraphics[scale=.25]{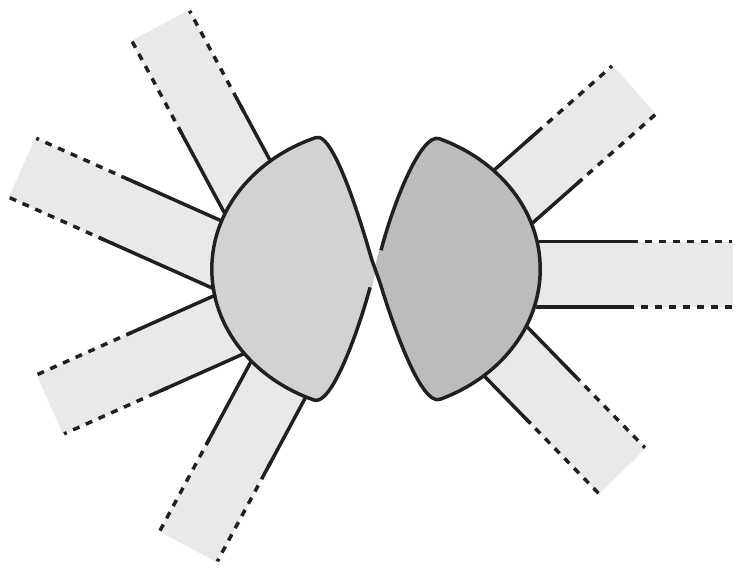}&\includegraphics[scale=.25]{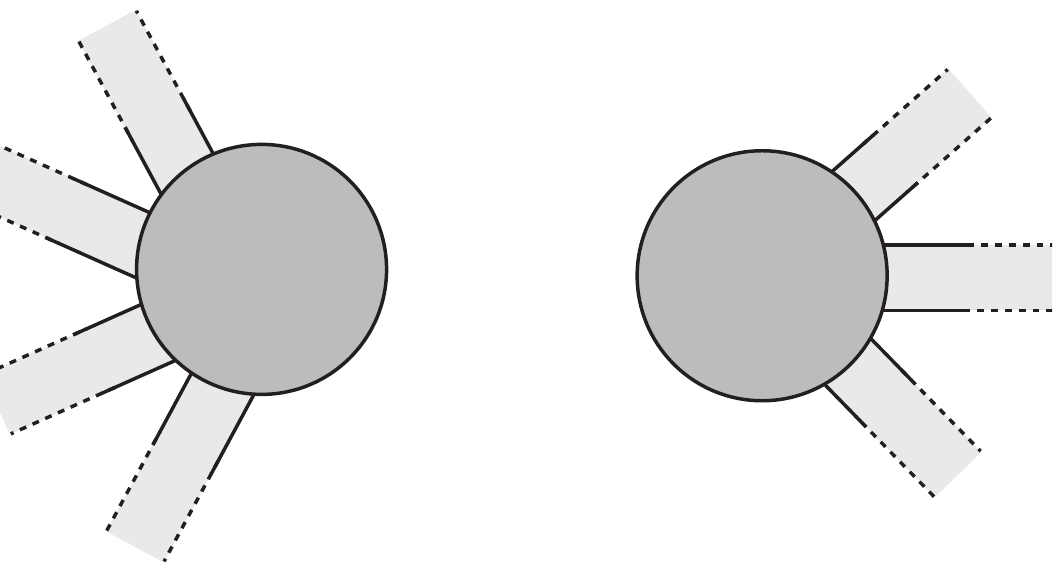} 
\\ \hline
\raisebox{6mm}{$G\tcon e$} 
 & \includegraphics[scale=.25]{ch4_37a}&\includegraphics[scale=.25]{ch4_38a}&\includegraphics[scale=.25]{ch4_37a}
\\ \hline
\end{tabular}
\caption{Contracting and Penrose-contracting an edge of a ribbon graph}
\label{tablecontractrg}
\end{table}

The \emph{partial petrial} of $G$ formed with respect to $e$, introduced in \cite{MR2869185}, is the ribbon graph $G^{\tau(e)}$ obtained from $G$ by  detaching an end of $e$ from its incident vertex $v$ creating arcs $[a,b]$ on $v$, and $[a',b']$ on $e$  (so that $G$ is recovered by identifying $[a,b]$ with $[a',b']$), then reattaching the end  by identifying the arcs antipodally (so that $[a,b]$ is identified with $[b',a']$). The result of this process is indicated in Figure~\ref{fig.pp}.  The \emph{Petrie dual} is the ribbon graph obtained by forming the partial petrial with respect to every edge (in any order).
\begin{figure}[t]
\centering
\begin{tabular}{ccc}
\includegraphics[height=12mm]{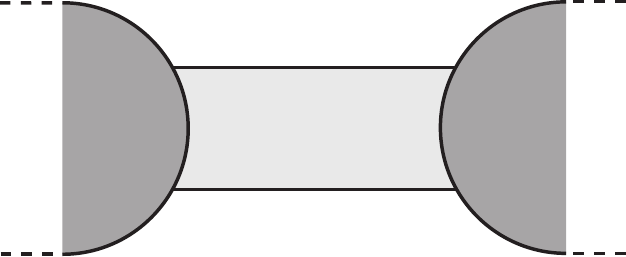}  &
\raisebox{4mm}{\includegraphics[width=15mm]{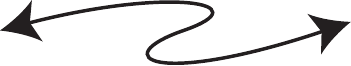}} &
\includegraphics[height=12mm]{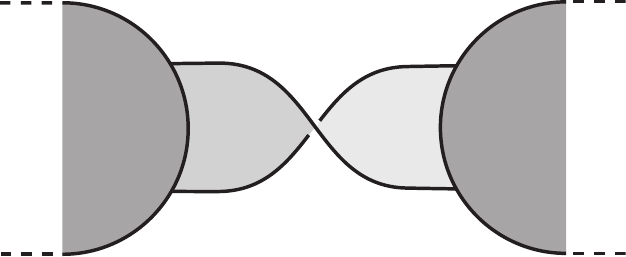}  \\
$G$ & $\longrightarrow$ & $G^{\tau(e)}$  \\ 
$G^{\tau(e)}$ &$\longleftarrow$ & $G$ 
\end{tabular}
\caption{Forming a partial Petrial at an edge of a ribbon graph}
\label{fig.pp}
\end{figure}

\begin{definition}\label{lk}
For a ribbon graph $G$ and edge $e$, we let  $G\tcon e$ denote the ribbon graph $G^{\tau(e)}/e$.  We call the operation this defines \emph{Penrose-contraction}. See Figure~\ref{tablecontractrg}.
\end{definition}

Deletion, contraction and Penrose-contraction are standard ribbon graph operations. In order to understand $\Omega_k (G)$ we introduce a new operation on a ribbon graph. For closure under this operation we need to augment the class of  graphic objects we consider. 

\begin{definition}\label{dty}
An \emph{edge-point ribbon graph} is an object obtained from a ribbon graph $G=(V,E)$ by contracting  each edge in some (possibly empty) subset $B$ of $E$ to a point. We call the points created by such a process \emph{singular points}. The image of edges under the contraction that are not singular points are called \emph{edges}. Its \emph{pinched-vertices} are components of the images of vertices (including the singular points) under the contraction.
\end{definition}
Figure~\ref{fig.krg} shows a edge-point ribbon graph. 
Note that the class of ribbon graphs is properly contained in the class of edge-point ribbon graphs. This is since a ribbon graph can be regarded as an edge-point ribbon graph that has no singular points. 

\begin{figure}[ht]
\begin{center}
\includegraphics[scale=.35]{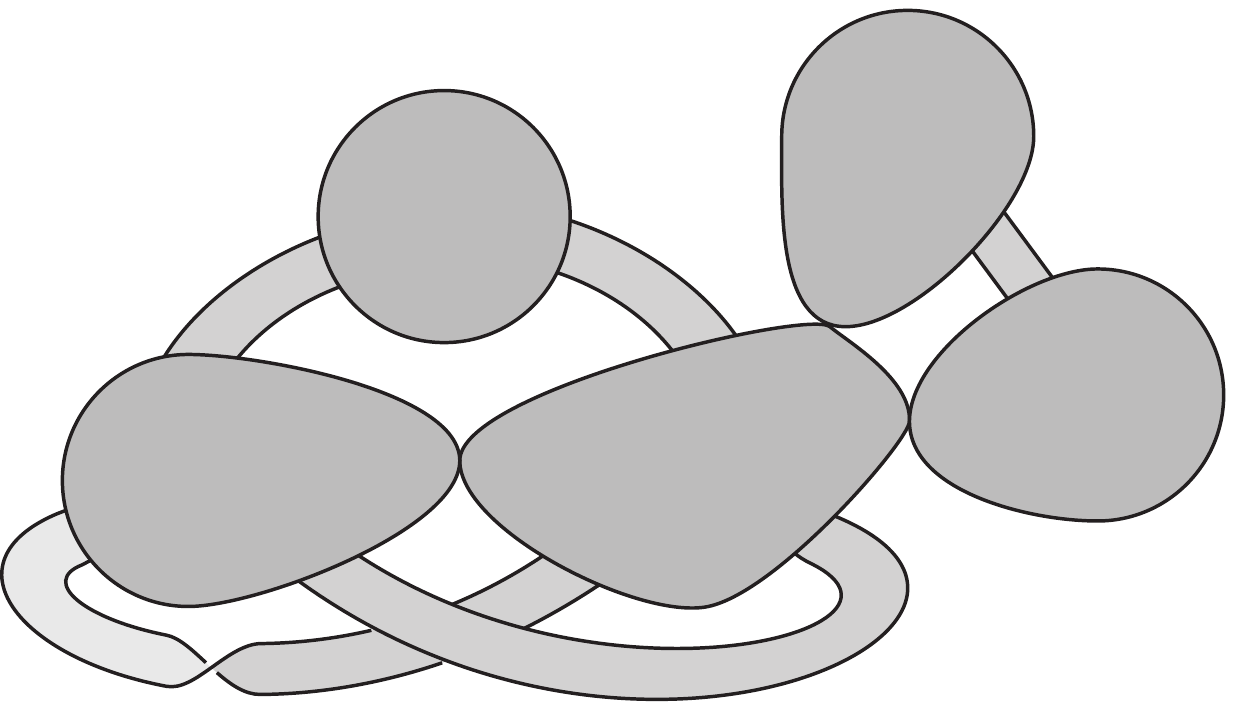} 
\hspace{10mm}
\includegraphics[scale=.35]{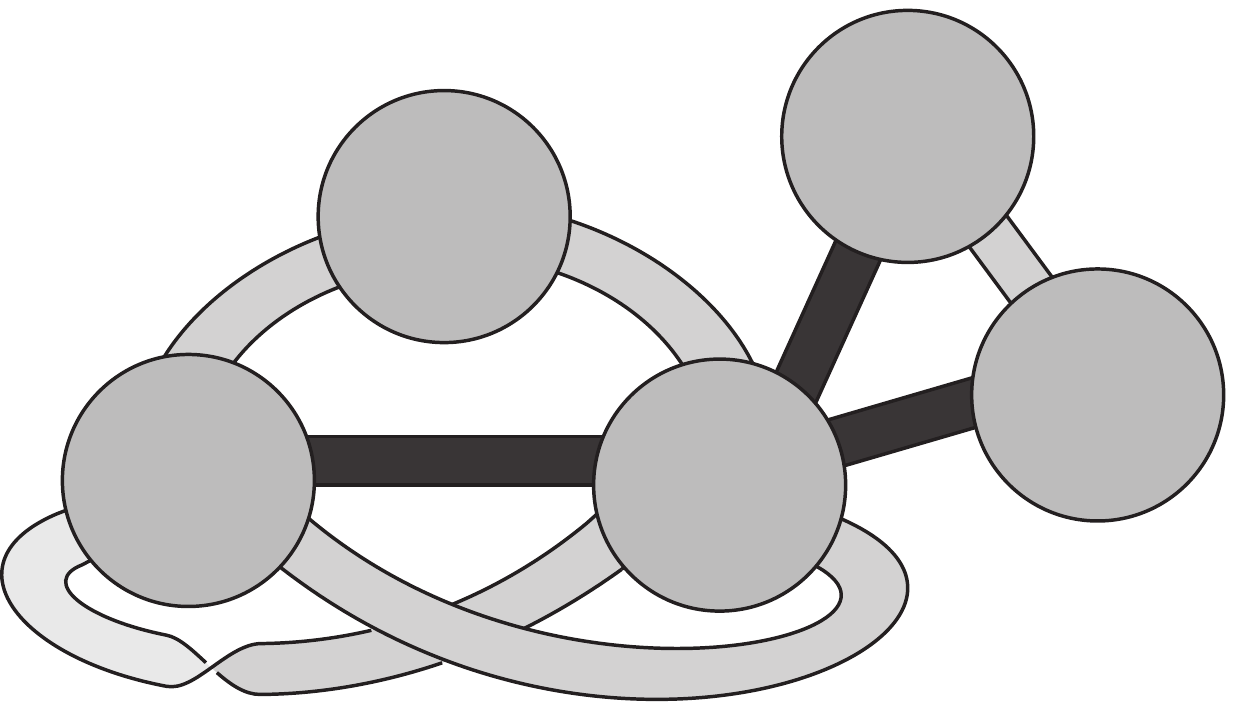}
\end{center}
\caption{An edge-point ribbon graph (left) and its description using edge colours (right) }
\label{fig.krg}
\end{figure}

A edgeless component of an edge-point ribbon graph is called a \emph{isolated vertex}.

Edge-point ribbon graphs can also be viewed as edge 2-coloured ribbon graphs where edges of one colour, by convention here dark grey, represent the edges that are contracted to a point in the formation of an edge-point ribbon graph, and  edges of the other colour (here light grey) correspond to the edges of the edge-point ribbon graph. See Figure~\ref{fig.krg}.

Two edge-point ribbon graphs are \emph{equivalent} if there is a homeomorphism (which is orientation preserving when the edge-point ribbon graphs are orientable) from one to the other that sends edges to edges, and pinched-vertices to pinched-vertices.
In the edge 2-colour presentation, they are equivalent if one is equivalent to a partial Petrial of the other, where the partial Petrial is formed with respect to a (possibly empty) subset of the (dark grey) edges that are contracted to a point.

The operations of deletion, contraction, partial Petriality, and Penrose-contraction for edge-point ribbon graphs are inherited from the ribbon graph operations. 
The extensions of deletion and partial Petriality require no comment. Contraction is defined by using the edge 2-coloured model and contracting in that using standard ribbon graph contraction. The extension of Penrose-contraction then follows.
Note that these operations may not be applied to the dark grey edges when using the edge 2-coloured model.

\begin{definition}\label{ftyh}
Let $G=(V,E)$ be an edge-point ribbon graph, and $e\in E$. Then $G\ident e$ denotes the edge-point ribbon graph obtained by  contracting $e$ to a point. See Figure~\ref{fig.iden}. We call the operation  $G\ident e$ defines, \emph{contraction to a point}. 
\end{definition}

\begin{figure}[ht]
\centering
\begin{tabular}{ccc}
\includegraphics[scale=.45]{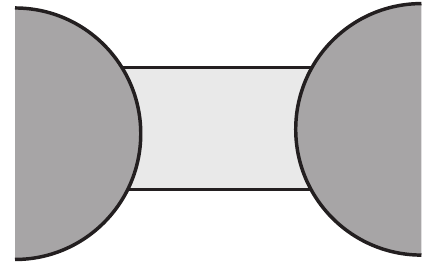} &  ~~~~~~\includegraphics[scale=.45]{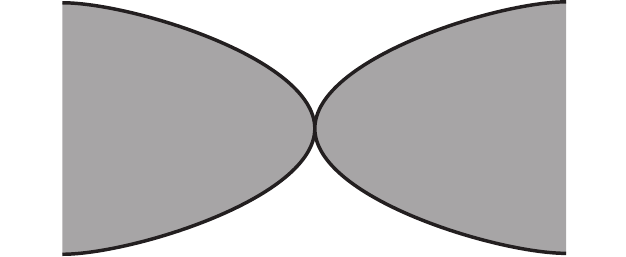} ~~~~~~&   \includegraphics[scale=.45]{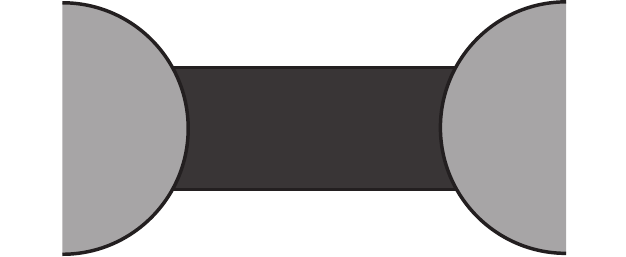}  \\
An edge $e$ of $G$ &    $G\ident e$ &  $G\ident e$ using edge colours
\end{tabular}
\caption{Forming $G\ident e$ at an edge  of an edge-point ribbon graph}
\label{fig.iden}
\end{figure}

Observe that the change in an edge-point ribbon graph made by the operations of deletion, contraction, Penrose-contraction, and contraction to a point  is local to the edge it is applied to. It follows that if the four operations commute when they are applied to different edges.   

For a ribbon graph $G$ and for $A\subseteq E(G)$ we let $G\ba A$ denote the result of deleting each edge in $A$. 
The notation $G/A$, $G\tcon A$ and $G\ident A$ is defined similarly.

We let  $\kappa(G)$ denote the number of connected components of an edge-point ribbon graph, and $\partial(G)$ denote its number of boundary components. 
Note that the boundary components of an edge-point ribbon graph need not be homeomorphic to a circle (unless it is also a ribbon graph). 
For example, the ribbon graph that is a plane theta-graph (i.e., the ribbon graph drawn on the plane with two vertices and three parallel edges between them) has three boundary components. Contracting any edge to a point results in an edge-point ribbon graph with two boundary components, one of which is homeomorphic to a circle, the other to a wedge of two circles (i.e., a figure-of-eight). 

If $G$ is an edge-point ribbon graph then it is important to remember that $\partial(G)$, in general, is \emph{not} equal to the to the number of boundary components of an edge 2-coloured ribbon graph that represents it.

\begin{remark}
Let us say a few words about what motivated  contraction to a point, and edge-point ribbon graphs.
Recent advances in understanding  the Bollob\'as--Riordan polynomial~\cite{MR1851080,MR1906909} and the Krushkal polynomial~\cite{Kr}, which are extensions of the Tutte polynomial to the setting of graphs in surfaces, have  been achieved by considering  
more exotic notions of deletion and contraction for graphs in surfaces, and extending the domains of the polynomials by requiring that they are closed under them.
 (See~\cite{EMMlv,HM,MS}.) In particular, in these extended domains the graph polynomials have `full' recursive deletion-contraction definitions that terminate in edgeless graphs, whereas they do not in their original restricted domains.

Contraction to a point, and edge-point ribbon graphs fit into this narrative. The Bollob\'as--Riordan and  Krushkal polynomials arise from considering a contraction operation on graphs in surfaces that contracts an edge to a point (as explained in~\cite{HM}). Analogously, contraction to a point, and edge-point ribbon graphs can be regarded as the structures that arise by contracting a 1-band in a band-decomposition (equivalently, an edge in a cellularly embedded ribbon graph)  to a point. This perspective also indicates why, in Definition~\ref{ght} below, it is natural to insist that  in a $k$-valuation at a singular vertex all edges incident to a singular vertex are of the same colour since everything is identified at the singular point. 
\end{remark}

\section{A graph polynomial $\Omega$}

We now use the above four operations on edge-point ribbon graphs to define a polynomial invariant of edge-point ribbon graphs.  
\begin{definition}\label{d.3}
Let $\Omega(G)\in \mathbb{Z}[w,x,y,z,t]$ be a  polynomial of edge-point ribbon graphs  recursively defined by
\[{\Omega}(G)= w\,{\Omega}(G\ident e)  + x\,{\Omega}(G/e) + y\,{\Omega}(G\ba e) + z\,{\Omega}(G\tcon e),  \]
and when $G$ is edgeless,
\[ {\Omega}(G) = t^{\kappa(G)} , \]
where  $\kappa(G)$ denotes the number of connected components of an edge-point ribbon graph.  
\end{definition}

\begin{example}
Using the edge 2-colour notation for edge-point ribbon graphs, we have the following.
\begin{align*}
\Omega\left(  \raisebox{-4mm}{\includegraphics[height=10mm]{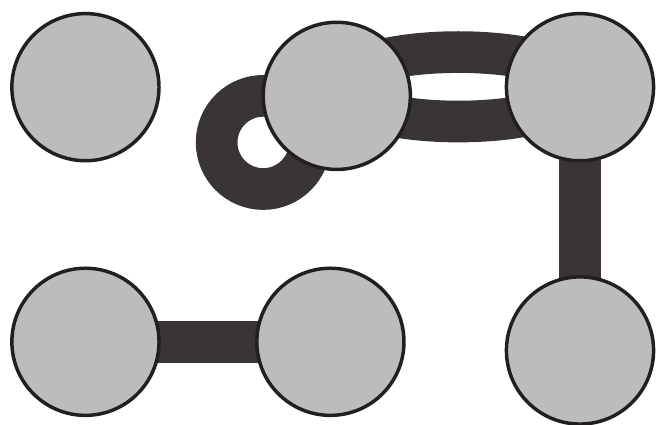}}   \right) &= t^3  
\\ \Omega\left(  \raisebox{-4mm}{\includegraphics[height=10mm]{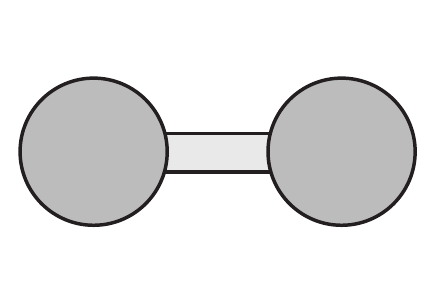}}   \right) &=   (w+x+z)t+yt^2
\\
\Omega\left(  \raisebox{-4mm}{\includegraphics[height=10mm]{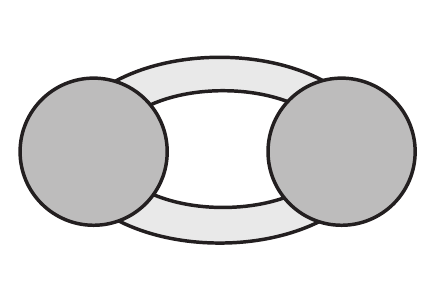}}   \right)  &= w(w+x+y+z)t +  x(w+y+z)t  +y(w+x+z)t \\ & \qquad\qquad\qquad\qquad\qquad\qquad+ z(w+x+y)t +  (x^2+y^2+z^2)t^2 
\\
\Omega\left(  \raisebox{-4mm}{\includegraphics[height=10mm]{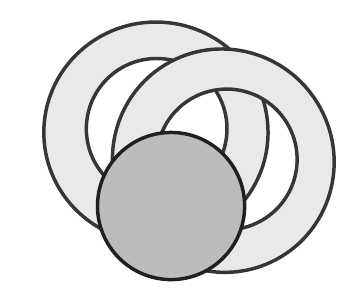}}   \right) &=  w(w+x+y+z)t +  x(w+x+z)t +y(w+y+z)t \\ & \qquad\qquad\qquad\qquad\qquad\qquad+ z(w+x+y)t +  (2xy+z^2)t^2 
\end{align*}
\end{example}

Definition~\ref{d.3} does not a priori result in a well-defined polynomial invariant since its value could, in principle, depend upon the order of edges to which the recursion relation is applied. The following theorem shows that the value of $\Omega(G)$ is in fact independent of such a choice and hence Definition~\ref{d.3} does give a well-defined polynomial invariant.  

\begin{theorem}\label{qws}
Let $G$ be an edge-point ribbon graph with edge set $E$. Then 
\begin{equation}\label{dh}
  \Omega(G) =\sum_{\substack{ (A,B,C,D)    \text{ an} \\ \text{ordered partition of }  E}}    w^{|A|} x^{|B|} y^{|C|} z^{|D|}  t^{ \partial (  G[A,B,C,D]   )  }   ,   
  \end{equation}
where 
\[ G[A,B,C,D]  :=  ((((G\ident A) /B) \ba C) \tcon D),   \]
and $\partial (  G[A,B,C,D]   )$ is its number of boundary components.
\end{theorem}
\begin{proof}
Recall that the operations of deletion, contraction, Penrose-contraction, and contraction to a point commute when they are applied to different edges. 
We use induction on the number of edges to prove that the state sum in \eqref{dh} satisfies the identities in Definition~\ref{d.3}. For clarity, in this proof we let $\Theta (G) $ denote the sum in the right-hand side of~\eqref{dh}.  

If $G$ has no edges then $\Theta (G) =  t^{\partial(G)} = t^{\kappa(G)} $, and so  $\Theta (G) =\Omega(G)$. Otherwise, for any edge $e$, we can write 
\begin{multline}\label{yh}
  \Theta (G)   =   \sum_{\substack{ (A,B,C,D)    \text{ ordered} \\ \text{partition of }  E(G) \\ \text{where } e\in A }}  \theta(G; A,B,C,D ) 
+\sum_{\substack{ (A,B,C,D)    \text{ ordered} \\ \text{partition of }  E(G) \\ \text{where } e\in B }}  \theta(G;  A,B,C,D ) 
\\
+\sum_{\substack{ (A,B,C,D)    \text{ ordered} \\ \text{partition of }  E(G) \\ \text{where } e\in C }}  \theta(G;  A,B,C,D )
 +\sum_{\substack{ (A,B,C,D)    \text{ ordered} \\ \text{partition of }  E(G) \\ \text{where } e\in D }}  \theta( G; A,B,C,D ) ,
  \end{multline}
where $\theta( G; A,B,C,D ):= w^{|A|} x^{|B|} y^{|C|} z^{|D|}  t^{ \partial (  G[A,B,C,D]   )  }$.

Focussing on the first sum in the right-hand side of \eqref{yh}, we see that, since $e\in A$,  we have $G[A,B,C,D] = (G\ident e)  [A\ba \{e\} ,B,C,D]$, and so  
\[ 
   \sum_{\substack{ (A,B,C,D)    \text{ ordered} \\ \text{partition of }  E(G) \\ \text{where } e\in A }}  \theta(G; A,B,C,D )  
=
  w\,  \sum_{\substack{ (A,B,C,D)    \text{ ordered} \\ \text{partition of }  E( G\ident e )  }}  \theta(G\ident e; A,B,C,D )  
= w\, \Omega(G\ident e),
  \]
where the last equality is by the inductive hypothesis.

Similar arguments, and making use of the fact that the four operations commute when applied to different edges, give that the second, third, and fourth sums in   \eqref{yh}, equal $x{\Omega}(G/e)$,  $y{\Omega}(G\ba e)$, and  $z{\Omega}(G\tcon e)$, respectively. It follows that $\Theta (G) =\Omega(G)$.
\end{proof}

Theorem~\ref{qws} immediately gives the following result.
\begin{corollary}\label{dsi}
$\Omega$ is well-defined in the sense that it is independent of the order of edges to  which the  recursion relations of Definition~\ref{d.3} are applied.
\end{corollary}

We note that $\Omega$ subsumes several graph polynomials from the literature. In the following proposition, $T(G)$ denotes the Tutte polynomial, $R(G)$ the Bollob\'as--Riordan polynomial,   $P(G)$  the Penrose polynomial, and $Q(G)$ the topological transition polynomial. 
\begin{proposition}\label{p.2}
Let $G$ be a ribbon  graph. Then
\begin{enumerate}
\item $\Omega\big(G ;0,\sqrt{y/x}, 1, 0, \sqrt{xy}\big) = x^{\kappa (G)}  (\sqrt{y/x})^{|V(G)|}   T(G;  x+1, y+1)$, when $G$ is plane,
\item $\Omega\big(G ;0,\sqrt{y/x}, 1, 0, \sqrt{xy}\big) = x^{\kappa (G)}  (\sqrt{y/x})^{|V(G)|}   R(G;  x+1, y, 1/\sqrt{xy})$,
\item $ \Omega(G ;0,1, 0, -1, \lambda)    = P(G; \lambda)$,
\item  $ \Omega(G;0,\alpha,\beta,\gamma,t) = Q(G; \left(\alpha,\beta,\gamma\right),t)  $.
\end{enumerate}
\end{proposition}
\begin{proof}
The result for the transition polynomial, $Q(G)$  is immediate. The remaining results follow by expressing $T(G)$, $R(G)$, and $P(G)$ in terms of the transition polynomial, as in \cite{MR2994409, MR2821551}.
\end{proof}

\section{Evaluations and interpretations}
In Section~\ref{dyh} we described how to construct the medial graph  of a graph embedded in a surface. Medial graphs can also be constructed from ribbon graphs. Let $G$ be a ribbon graph. Construct an embedded graph $G_m$ by taking one point in each edge of $G$ as the vertices of $G_m$. To construct the edges,  from each vertex of $G_m$ draw four non-intersecting curves on the edge from that vertex to the four ``corners'' of  the edge (i.e., the four end-points of the arcs where the edge intersects its incident vertices). Connect these curves up by following the boundary of $G$ around the vertices. Finally, for each isolated vertex, add a free-loop that follows its boundary.

The \emph{medial graph of an edge-point ribbon graph} $G$ is the 4-regular embedded graph $G_m$ constructed by following the above procedure for ribbon graphs, and then adding a vertex  at each singular point.  See Figure~\ref{mdg} for an example.  The medial graph $G_m$ has two types of vertex: those corresponding to edges of $G$, we call these \emph{non-singular vertices}; and those corresponding to singular points of $G$, we call these \emph{singular vertices}.  By convention, here we draw the singular vertices of $G_m$ as hollow dots.

\begin{figure}[ht]
\[
  \includegraphics[scale=.4]{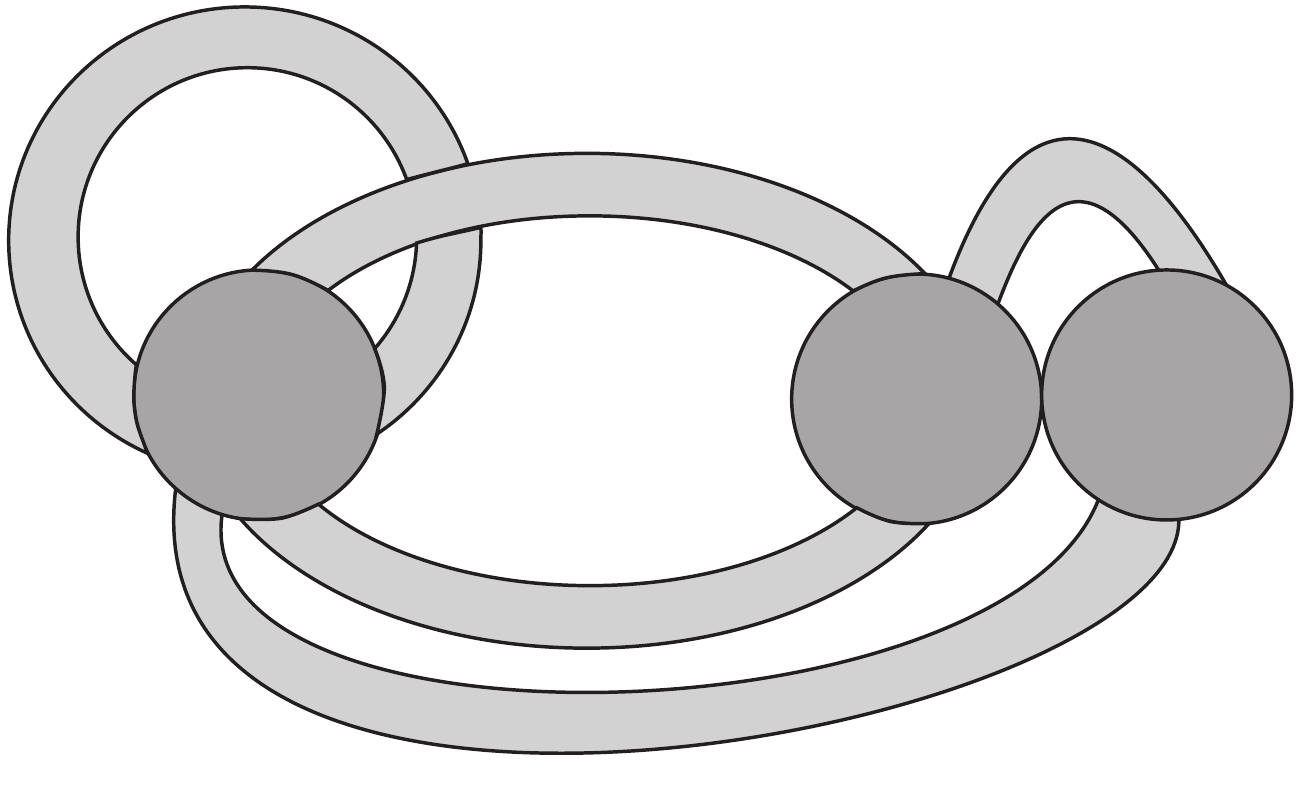}   \hspace{15mm}   \includegraphics[scale=.4]{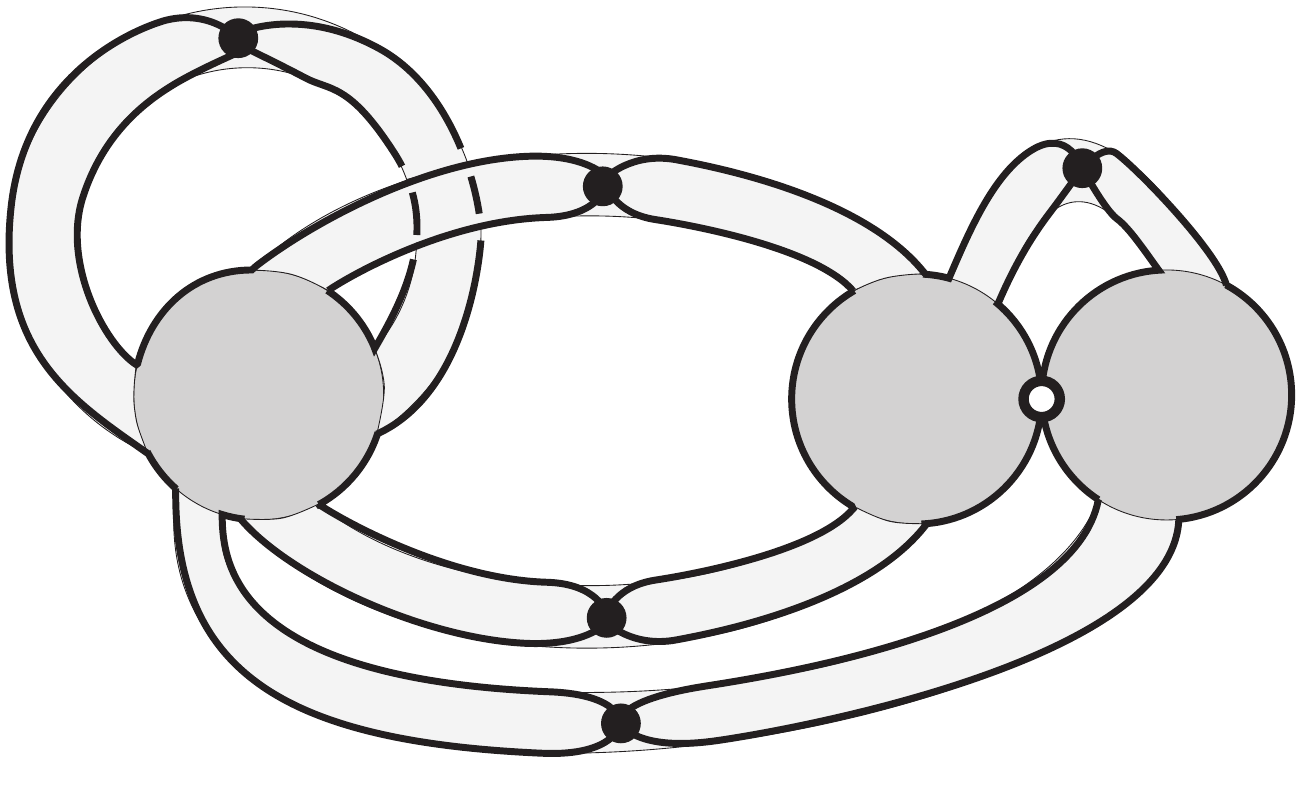} 
\]
\caption{An edge-point graph $G$ (left) and its medial graph $G_m$ embedded in $G$ (right)}
\label{mdg}
\end{figure}

When using the edge 2-colour model of an edge-point ribbon graph, $G_m$ can be formed by taking its medial graph of the ribbon graph  and declaring the vertices on the dark grey edges to be singular. 
 
Medial graphs of an edge-point ribbon graph admit a \emph{canonical checkerboard colouring}. This is obtained by colouring a region of the surface it lies in black if it contains a vertex of $G$ and white otherwise. The notion of a $k$-valuation extends to medial graphs of  edge-point ribbon graphs by insisting that all edges incident to a singular vertex are of the same colour (to reflect that everything is identified at a singular point), as follows.

\begin{definition}\label{ght}
Let $k$ be a natural number, and $G$ be an edge-point ribbon graph. 
A  \emph{$k$-valuation} of $G_m$  is an edge $k$-colouring of $G_m$ such that each vertex is incident to an even number (possibly zero) of edges of each colour, and all edges incident to a singular vertex are of the same colour.  

If $G_m$ is canonically checkerboard coloured then the four possible configurations of colours about a vertex, are \emph{white}, \emph{black}, \emph{crossing}, or \emph{total},  as described in Figure~\ref{f.kval}. Singular vertices are always total. 
 We let $\total(\phi)$ denote the number of non-singular total vertices in a $k$-valuation $\phi$.
\end{definition}

\begin{theorem}\label{t.1}
 Let $G$ be an edge-point ribbon graph, and let $k\in \mathbb{N}$. 
Then 
\begin{equation}\label{po}
\Omega(G;w,x,y,z,k)= \sum_{\phi \text{ a $k$-valuation of $G_m$}}  (w+x+y+z)^{\total(\phi)} x^{\white(\phi)} y^{\black(\phi)}z^{\cross(\phi)}.
\end{equation}
\end{theorem}
\begin{proof}
We use induction on the number of edges of an edge-point ribbon graph $G$. 
The claim  is easily verified  when $G$ has no edges.

Now assume that $G$ has edges and that the claim holds for all edge-point ribbon graphs with fewer edges that $G$.
For  a $k$-valuation $\phi$ of $G_m$, and for $U\subseteq V(G_m)$, set
\[  \omega(G,\phi, U) :=  (w+x+y+z)^{\total(\phi, U)}  x^{\white(\phi,U)}  y^{\black(\phi,U)}  z^{\cross(\phi,U)},  \]
where $\total(\phi, U)$, $\white(\phi,U)$, $\black(\phi,U)$,   $\cross(\phi,U)$ denote the numbers of total, white, black, and crossing vertices contained in  $U$ in the $k$-valuation $\phi$  of $G_m$.

Fix an edge $e$  of $G$, and let $v_e$ be its corresponding vertex in $G_m$. 
By separating the sum according to what the $k$-valuation does at $v_e$, we can write the sum on the right-hand side of~\eqref{po} as
\begin{multline}\label{dtg}
w  \Big( \sum_{  \substack{\phi \text{ a $k$-val. of }    G_m \\  \phi\text{ tot. at  } v_e   }} \omega(G,\phi, V(G_m)\backslash \{v_e\}) \Big)
+
x  \Big( \sum_{  \substack{\phi \text{ a $k$-val. of }    G_m \\  \phi\text{ tot. or wh. at } v_e   }} \omega(G,\phi, V(G_m)\backslash \{v_e\}) \Big)
\\
+
y  \Big( \sum_{  \substack{\phi \text{ a $k$-val. of }    G_m \\  \phi\text{ tot. or bl. at } v_e   }} \omega(G,\phi, V(G_m)\backslash \{v_e\}) \Big)
+
z  \Big( \sum_{  \substack{\phi \text{ a $k$-val. of }    G_m \\  \phi\text{ tot. or cr. at } v_e   }} \omega(G,\phi, V(G_m)\backslash \{v_e\}) \Big)
\end{multline}

We show that the sum in the first term in \eqref{dtg} equals $Q(G\ident e )$.
Consider  $G$ locally at an edge $e$, and $G_m$ locally at the corresponding vertex $v_e$ as shown in Table~\ref{p.kval.f1}. 
\begin{table}[ht]
\centering
\begin{tabular}{|c|c|c|c|c|c|}\hline
  & $G$ & $G\ident e$ & $G/e$ & $G\ba e$   & $G\tcon e$\\ \hline
 \raisebox{5mm}{$G$} & 
 \labellist \small\hair 2pt
\pinlabel {$e$}  at 64 39
\endlabellist\includegraphics[scale=.4]{ch1_9m}  & \includegraphics[scale=.4]{medq} &\includegraphics[scale=.4]{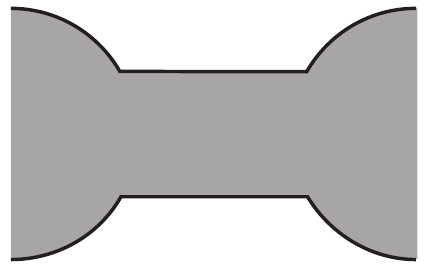}& \includegraphics[scale=.45]{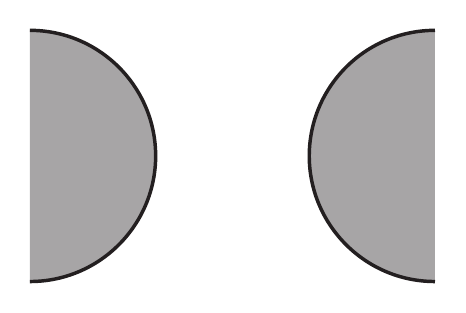}& \includegraphics[scale=.45]{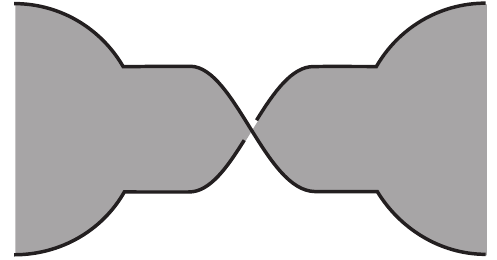} \\ \hline
  \raisebox{5mm}{$G_m$} &\labellist \small\hair 2pt
\pinlabel {$v_e$}  at 78 20
\endlabellist \includegraphics[scale=.35]{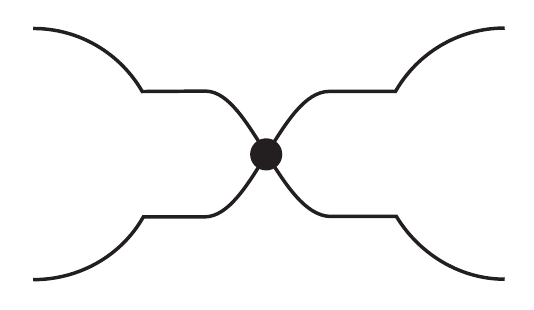}& \includegraphics[scale=.35]{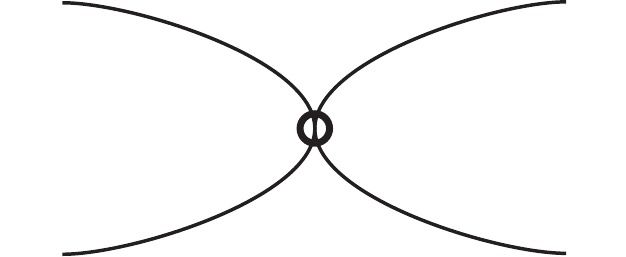}  &\includegraphics[scale=.35]{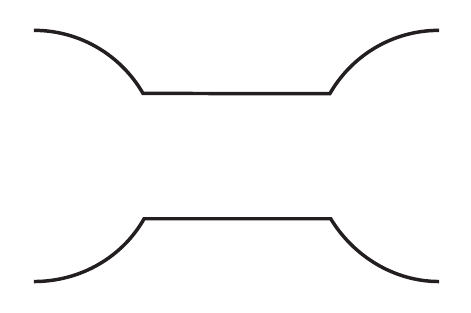}&\includegraphics[scale=.35]{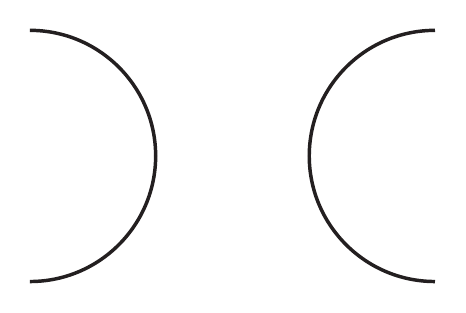}&  \includegraphics[scale=.4]{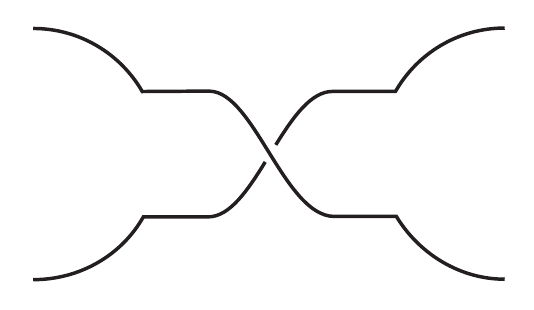} \\ 
\hline
\end{tabular}
\caption{Local differences in $G$, $G\ident e$, $G/e$, $G\ba e$, $G\tcon e$, and their medial graphs for a non-loop edge $e$}
\label{p.kval.f1}
\end{table}
Note that in the table we are not assuming that the vertices at the ends of $e$ are distinct or that they lie on the plane on which they are drawn, and so Table~\ref{p.kval.f1} and the following argument also includes the cases when $e$ is an orientable or non-orientable loop. 
Table~\ref{p.kval.f1} also shows $G\ident e$ and  $(G \ident e)_m$ at the corresponding locations. 
All the embedded graphs in the table are identical outside of the region shown. 
Let $\phi'$ be a $k$-valuation of $(G \ident e)_m$. 
Then the arcs of  $(G \ident e)_m$ shown in Table~\ref{p.kval.f1} are all coloured with the same element $i$.
  The $k$-valuation $\phi'$ of $(G\ident e)_m$ naturally induces a $k$-valuation $\phi$ of $G_m$ in which $(\phi,v_e)$ is total. As this process is reversible, we have a bijection between   the set of all $k$-valuations of  $(G \ident e)_m$, and the set of all $k$-valuations of $G_m$ in which $v_e$ is total. Thus
 \begin{align*}
 \sum_{  \substack{\phi \text{ a $k$-val. of }    G_m \\  \phi\text{ tot. } v_e   }} \omega(G,\phi, V(G_m)\backslash \{v_e\}) 
   &=  \sum_{ \phi  \text{ a $k$-val. of }    (G\ident e)_m  } \omega (G\ident e,\phi, V((G\ident e)_m) 
   \\& = \Omega (G\ident e; w,x,y,z, k  ),
 \end{align*}
where the second equality follows by the inductive hypothesis.

Similar arguments give that the second, third, and fourth sums in \eqref{dtg} equal $ \Omega (G/ e)$, $ \Omega (G\ba e)$  and $ \Omega (G\tcon e)$ respectively. Thus Equation \eqref{dtg}  equals $w{\Omega}(G\ident e)  + x{\Omega}(G/e) + y{\Omega}(G\ba e) + z{\Omega}(G\tcon e)$, which
 is   $\Omega(G)$.
\end{proof}

The significance of Theorem~\ref{t.1} is that it provides a recursive, skein-theoretic way to compute the  generating function $\Omega_k$ for the  $k$-valuations of Definition~\ref{d.2}.
\begin{corollary}\label{sf}
Let $G$ be a ribbon graph or a cellularly embedded graph,  and let $k$ be a natural number. Then 
\[ \Omega_k (G; w, x,y,z) = {\Omega}(G;(w -x-y-z) ,x,y,z,k) . \] 
In particular,  $\Omega_k (G; w, x,y,z)$ is a polynomial in $k$.
\end{corollary}

The \emph{(geometric) dual}, $G^*$ of a ribbon graph $G=(V(G),E(G))$ is  constructed as follows. Recalling that, topologically, a ribbon graph is a surface with boundary, we cap off the holes using a set of discs, denoted by $V(G^*)$, to obtain a surface without boundary. Then $G^*=(V(G^*),E(G))$ is the ribbon graph obtained by removing the original vertices. 

\begin{corollary}\label{gsi}
Let $G$ be a ribbon graph or an embedded graph. Then 
\[  {\Omega}(G; -2,1,0,1,\lambda)= \sum_{A\subseteq E(G)} \chi((G^{\tau(A)})^*;\la)  ,\]
where $\chi(H;\la)$ is the chromatic polynomial of  $H$.
\end{corollary}
\begin{proof}
By Theorem~\ref{t.1}, for each $k$,  ${\Omega}(G; -2,1,0,1,\lambda)$ counts the number of $k$-valuations of $G_m$ in which each vertex is either white or crossing. We restrict to this type of  $k$-valuation in the remainder of the proof.

A {\em proper boundary $k$-colouring } of a ribbon graph is a map from its set of boundary components to $[k]$ with the property that whenever  two boundary components share a common edge, they are assigned different colours.

Let $A_{\phi}\subseteq E(G)$ be the set of edges corresponding to vertices of $G_m$ with crossings in the  $k$-valuation $\phi$.  The cycles in $G_m$   determined by the colours  in the $k$-valuation $\phi$ follow exactly the boundary components of the partial Petrial $G^{\tau(A_{\phi})}$. Moreover, the colours of the  cycles in the $k$-valuation induce a colouring of the boundary components of $G^{\tau(A_{\phi})}$. It is easy to see that this establishes a 1-1 correspondence between $k$-valuations of $G_m$ in which each vertex is either white or crossing, and proper boundary colourings of $G^{\tau(A_{\phi})}$. However, since the boundary components of a ribbon graph $H$ correspond with the vertices of its dual $H^*$, it follows that there is a  1-1 correspondence between proper boundary colourings of $G^{\tau(A_{\phi})}$, and proper vertex colourings of $(G^{\tau(A_{\phi})})^*$.  It follows that $\Omega_k (G; 0,1,0,1)  =  \sum_{A\subseteq E(G)} \chi((G^{\tau(e)})^*;k)$. Since this is true for each natural number $k$, and using Corollary~\ref{sf}, the result about    $ \Omega(G; -2,1,0,1,\lambda)$ follows.
\end{proof}

It is interesting to compare Corollary~\ref{gsi} with Theorem~5.3 of \cite{MR3326433} which gives that $  \Omega(G;0,1,0,-1,\la)= \sum_{A\subseteq E(G)}  (-1)^{ |A|}  \chi ((   G^{\tau(A)}   )^*   ;\lambda)$.

\begin{corollary}\label{awe1}
Let $G$ be a plane ribbon graph. Then 
\[ {\Omega}(G; -2,1,0,1,\lambda) = P(G;\la)  ,\]
where $P(G;\la)$ is the Penrose  polynomial.
\end{corollary}
\begin{proof}
This follows from Corollary~\ref{gsi} since it was shown in \cite{MR2994409} that $\sum_{A\subseteq E(G)} \chi((G^{\tau(A)})^*;\la)$ equals the Penrose polynomial of a plane graph.
\end{proof}

\begin{corollary}\label{awe2}
If $G$ is a cubic ribbon graph then 
\[  {\Omega}(G; -2,1,0,1,3) =\# \text{proper edge 3-colourings of } G.\]
\end{corollary}
\begin{proof}
This argument is an adaptation of the argument in \cite{MR0281657} used to show that the Penrose polynomial counts proper edge 3-colourings of plane graphs. Proper edge 3-colourings of a ribbon graph are in bijection with proper boundary 3-colourings of partial Petrials of that ribbon graph as indicated in Figure~\ref{stag}.  

\begin{figure}[ht]
\centering
\begin{tabular}{ccc}
 \labellist \small\hair 2pt
\pinlabel {$b$}  at 29 109
\pinlabel {$c$}  at 29 30
\pinlabel {$b$}  at 259 109
\pinlabel {$c$}  at 259 30
\pinlabel {$a$}  at 146 74
\endlabellist\includegraphics[scale=.5]{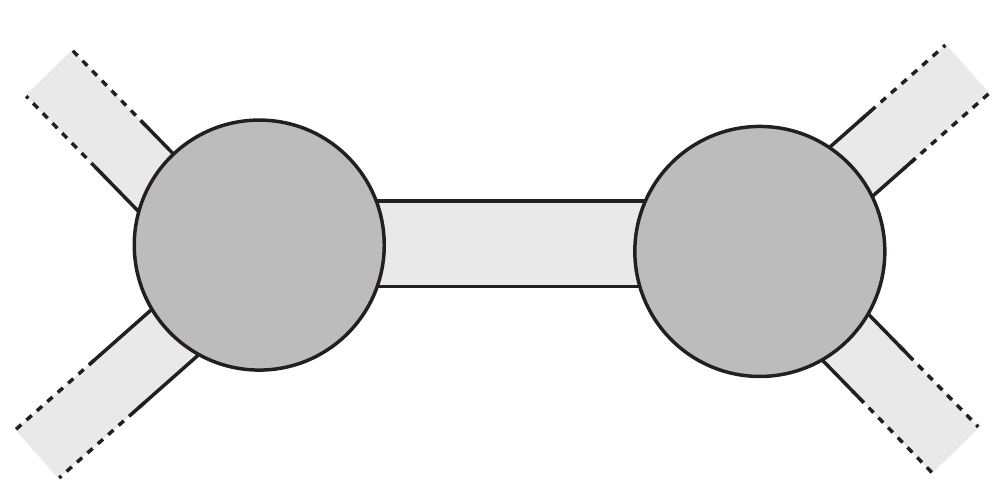} 
 & \raisebox{12mm}{$ \longleftrightarrow$}& 
  \labellist \small\hair 2pt
\pinlabel {$a$}  at 21 72
\pinlabel {$a$}  at 269 72
\pinlabel {$b$}  at 143 35
\pinlabel {$c$}  at 143 106
\endlabellist\includegraphics[scale=.5]{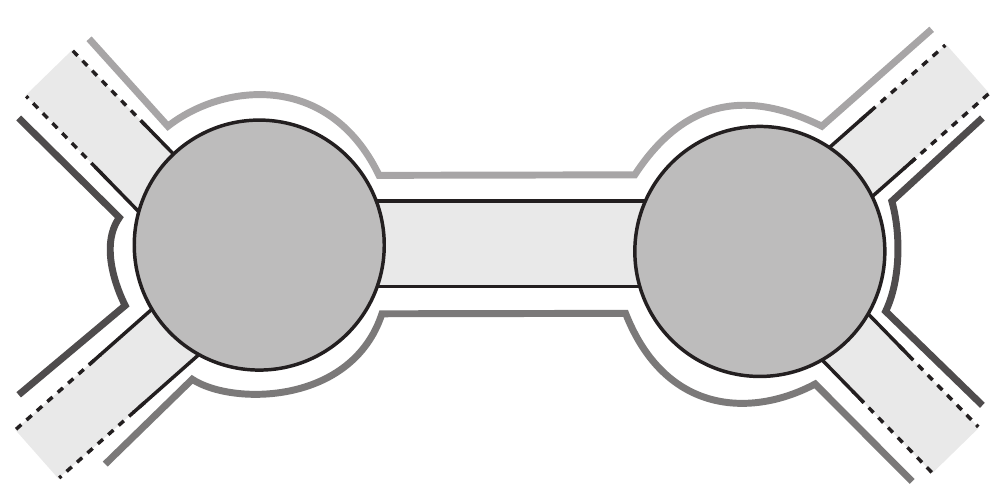} 
\\
 \labellist \small\hair 2pt
\pinlabel {$b$}  at 29 109
\pinlabel {$c$}  at 29 30
\pinlabel {$c$}  at 259 109
\pinlabel {$b$}  at 259 30
\pinlabel {$a$}  at 146 74
\endlabellist\includegraphics[scale=.5]{col1} 
 & \raisebox{12mm}{$ \longleftrightarrow$}& 
  \labellist \small\hair 2pt
\pinlabel {$a$}  at 21 72
\pinlabel {$a$}  at 269 72
\pinlabel {$b$}  at 143 35
\pinlabel {$c$}  at 143 106
\endlabellist\includegraphics[scale=.5]{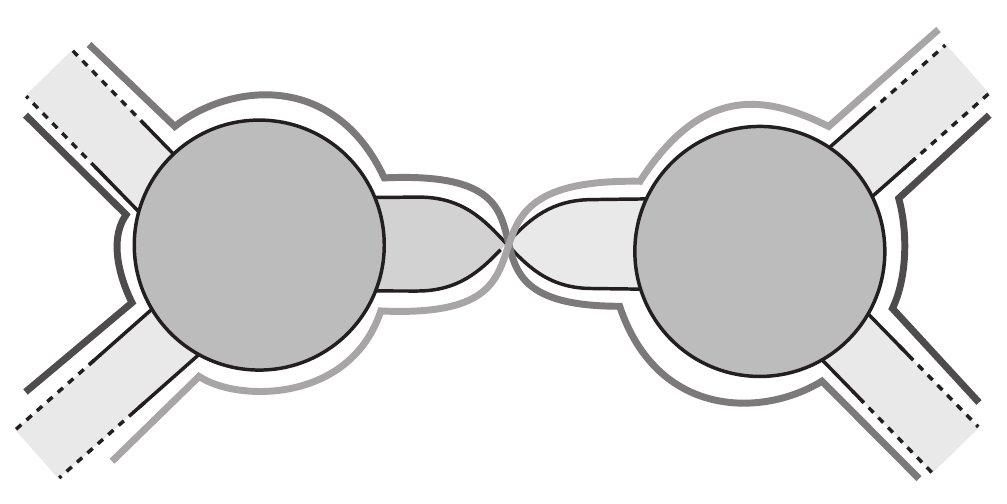} 
\end{tabular}
\caption{Moving between edge 3-colourings and boundary 3-colourings}
\label{stag}
\end{figure}

As  in the proof of Corollary~\ref{gsi}, $\Omega_3 (G; 0,1,0,1)$  counts the number of proper boundary 3-colourings of the partial Petrials of $G$. The result follows. 
\end{proof}

For plane graphs $G$, Corollary~\ref{awe1} implies that $\Omega(G; -2,1,0,1,\lambda)$ coincides with $P(G)$. This is not true, in general, when $G$ in non-plane so Corollary~\ref{awe1} does not give a new interpretation of the topological Penrose polynomial of \cite{MR2994404}.  

However, a direct consequence of Corollary~\ref{awe2} is that $\Omega(G; -2,1,0,1,\lambda )$  (and $\Omega(G; w,x,y,z,t )$)   offers an extension of the original plane Penrose polynomial of \cite{MR1428870,MR0281657} to a topological graph polynomial that counts edge 3-colourings of all (not just cubic) graphs, one of its key properties. With this in mind make the following definition.
\begin{definition}
Let $G$ be an edge-point ribbon graph. We call  
\[
P_p(G;\lambda) := \Omega(G; -2,1,0,1,\lambda )
\]
the \emph{pointed-Penrose} polynomial. 
\end{definition}
Note that  it satisfies the recursion
${P_p}(G)=  {P_p}(G/e)  + \,{P_p}(G\tcon e)-2\,{P_p}(G\ident e)  $.

It would be interesting to determine what other properties of the classical plane Penrose polynomial extend to $P_p$ (a catalog of some properties that do and do not extend to the topological Penrose polynomial may be found in \cite{EMMbook,MR2994409}).  Similarly, a new avenue of investigation would be comparing and contrasting the properties of the pointed Penrose polynomial $P_p$ and the topological Penrose polynomial $P$ for non-planar graphs.  We leave these as open problems.

\section*{Acknowledgements}
This paper arose from conversations between its three authors at the Dagstuhl Seminar 16241,\textit{Graph Polynomials: Towards a Comparative Theory}. They would like to thank Schloss Dagstuhl for providing a  stimulating and productive environment.

Ellis-Monaghan's work on this project was supported by NSF grant DMS-1332411

Louis H. Kauffman is supported by the Laboratory of Topology and Dynamics, Novosibirsk State University
(contract no. 14.Y26.31.0025 with the Ministry of Education and Science of the Russian Federation).

\end{document}